\documentclass{article}
\usepackage{geometry} %This changes the border size.
\usepackage{amsrefs}

\usepackage[utf8]{inputenc} % Required for inputting international characters
\usepackage[T1]{fontenc} % Output font encoding for international characters

\usepackage{amsmath, amsthm, amssymb, mathrsfs}
\usepackage{enumerate}
\usepackage{tikz-cd}
\usetikzlibrary{arrows}

\usepackage{hyperref}

\usetikzlibrary{matrix,arrows,decorations.pathmorphing}

 \let\temp\phi
\let\phi\varphi
\let\varphi\temp

\newcommand{\cl}[1]{\overline{#1}}

\newcommand{\C}{\mathbb{C}}

\newcommand{\N}{\mathbb{N}}

\DeclareMathOperator{\Span}{span}

\newcommand{\calB}{\mathcal{B}}

\newcommand{\calD}{\mathcal{D}}
\newcommand{\calE}{\mathcal{E}}

\newcommand{\calG}{\mathcal{G}}

\newcommand{\calJ}{\mathcal{J}}
\newcommand{\calK}{\mathcal{K}}

\newcommand{\calM}{\mathcal{M}}

\newcommand{\Angle}[1]{\left\langle #1 \right\rangle}

\renewcommand{\Hat}{\widehat}

\DeclareMathOperator{\tr}{tr}
\DeclareMathOperator{\diag}{diag}

\theoremstyle{plain}
\newtheorem{lemma}{Lemma}
\newtheorem{theorem}[lemma]{Theorem}

\newtheorem{proposition}[lemma]{Proposition}
\newtheorem{corollary}[lemma]{Corollary}

\theoremstyle{definition}

\newtheorem{example}[lemma]{Example}
\newtheorem{remark}[lemma]{Remark}
\newtheorem{definition}[lemma]{Definition}

\DeclareMathOperator{\vth}{\vartheta}

\title{Chromatic numbers and a Lov\'{a}sz type inequality for non-commutative graphs}
\author{Se-Jin Kim and Arthur Metha}
\date{}
\begin{document}
\maketitle

\begin{abstract}
Non-commutative graph theory is an operator space generalization of graph theory. Well known graph parameters such as the independence number and Lov\'{a}sz theta function were first generalized to this setting by Duan, Severini, and Winter \cite{DSW}.

We introduce two new generalizations of the chromatic number to non-commutative graphs and provide a generalization of the Lov\'{a}sz sandwich inequality. In particular, we show the chromatic number of the orthogonal complement of a non-commutative graph is bounded below by its theta number. We also provide a generalization of both Sabadussi's Theorem and Hedetniemi's conjecture to non-commutative graphs.
\end{abstract}

\section{Introduction}
Given a graph on $n$ vertices one can associate two different subspaces of the $n \times n$ matrices that encode all of the information of the graph. This has motivated the generalization of several well known graph theoretic concepts to a larger class of objects.

In \cite{DSW}, Duan, Severini, and Winter describe a version of non-commutative graph theory whose underlying objects consist of \emph{submatricial operator systems}. The aforementioned authors generalize the independence number and Lov\'{a}sz theta number to submatricial operator systems.

In \cite{St}, Stahlke works with a similar but distinct definition of a non-commutative graph. Instead of working with submatricial operator systems, Stahlke associates a subspace of matrices whose elements all have zero trace to a graph. Stahlke generalizes several classical graph theory concepts to these traceless subspaces including the chromatic number, clique number and notion of graph homomorphism.

Thus, there are two quite different  subspaces of matrices to associate to graphs that lead to two different ways to create a non-commutative graphs theory. In this paper we discuss both the \emph{submatricial operator system} and \emph{submatricial traceless self-adjoint operator space} definitions of a non-commutative graph.

There is currently no notion of the complement of a non-commutative graph that generalizes the graph complement. By working with both of the above definitions we are able to generalize the complement of a graph using the orthogonal complement with respect to the Hilbert-Schmidt inner product. We conclude this section by reviewing the definition of several non-commutative graph parameters and show that some of these parameters can be approximated by evaluating classical graph parameters.

In \cite{Lov} Lov\'{a}sz introduced his well known theta number of a graph, $\theta(G)$. Lov\'{a}sz shows that this number determines the following bounds on the independence number, $\alpha(G)$, and the chromatic number of the graph complement $\chi(\cl{G})$.
\begin{align*}
\alpha(G) \leq \theta(G) \leq \chi(\cl{G}).
\end{align*}
These two inequalities are often referred to as the Lov\'{a}sz sandwich theorem. In section \ref{Sand section} we establish this result for non-commutative graphs using new generalizations of the chromatic number to the non-commutative setting.

Given two graphs $G$ and $H$ the Cartesian product is the graph $G \Box H$ with vertex set $V(G) \times V(H)$ and edge relation given by $(v,a) \sim (w,b)$ if one of $v \sim_G w$ and $a =b$ or $v = w$ and $a \sim_H b$ holds. A Theorem of Sabidussi tell us $\chi(G \Box H) = \max\{\chi(G), \chi(H)\}$ for any $G$ and $H$. We introduce a Cartesian product and establish a generalization of this result for \emph{submatricial traceless self-adjoint operator spaces} in Section \ref{Sabidussi and Hedet section} . In section \ref{Sabidussi and Hedet section} we also establish a categorical product for submatricial traceless self-adjoint operator space and extend a Theorem of Hedetniemi to submatricial traceless self-adjoint operator spaces.

In \cite{DSW}, it is shown that that the independence number of a submatricial operator system is bounded above by its Lov\'{a}sz number. This provides the first inequality for a generalized Lov\'{a}sz sandwich theorem. We provide a generalization of the second inequality to non-commutative graphs.
We are also able to provide lower and upper bounds on the chromatic number introduced by Stahlke's in \cite{St}. We answer a question posed by Stahlke by generalizing the equation $\chi(G) \omega(\cl{G}) \geq n$ to non-commutative graphs. %Paulsen and Ortiz \cite{PO} show that graphs $G$ and $H$ are isomorphic if and only if the corresponding submatricial operator systems are completely order isomorphic. We show an analogous result for submatricial traceless self-adjoint operator spaces.

\subsection{Notation}
Let $M_n$ denote the vector space of $n \times n$ matrices over $\C$. This vector space can also be viewed as a Hilbert space using the inner product $\langle A, B \rangle :=\tr(B^*A)$. By a submatricial operator system, we mean a linear subspace $S$ of $M_n$ for which the identity matrix $I$ belongs to $S$ and for which $S$ is closed under the adjoint map $^*$. A submatricial traceless self-adjoint operator space is a linear subspace $\calJ \subset M_n$ for which $\calJ$ is closed under the adjoint operation $^*$ and for which given any $A \in \calJ$, the trace of $A$ is zero.

If $S$ and $T$ are two submatricial operator systems, then a linear map $\phi$ from $S$ into $T$ is called completely positive (cp) if for each positive integer $n$, and for each positive semi-definite matrix $X = [x_{i,j}]_{i,j}$ in $M_n(S)$, the matrix
\begin{align*}
  \phi^{(n)}(X):= \left[\phi(x_{ij}) \right]_{i,j}
\end{align*}
in $M_n(T)$ is positive semi-definite. We say that the cp map $\phi$ is unital and completely positive (ucp) if $\phi$ maps the identity $I$ to the identity $I$. We say that $\phi$ is completely positive and trace preserving (cptp) if $\tr(\phi(X)) = \tr(X)$ for all $X \in S$. For more on cp maps see \cite{PBook}.

We will also be using the following graph theory terminology. A graph $G=(V,E)$ is an ordered pair consisting of a vertex set $V$ and edge set $E \subset V \times V$. Since we are working with undirected graphs we require that if $(i_1, i_2) \in E$ then $(i_2, i_1) \in E$. We say vertices $i_1$ and $i_2$ are \emph{adjacent}, or connected by an edge, and write $i_1 \sim i_2$, whenever $(i_1, i_2) \in E$. An \emph{independent set} of a graph $G$ is a subset $v \subset V$ such that for any two distinct elements $i_1, i_2 \in v$ we have $i_1 \not\sim i_2$. For a graph with $n$ vertices it will be standard to consider the vertex set to be $V= \lbrace 1, \dots, n \rbrace$, which we will denote by $[n]$.

\section{Non-commutative graphs}\label{Non-commutative graphs basics}
A non-commutative graph is sometimes viewed as any submatricial operator system $S$. Non-commutative graphs have also been described as any submatricial traceless  self-adjoint operator space $\calJ$. In this section we review how one can view a classical graph as either of these objects without losing information about the graph itself. We also discuss several parameters for non-commutative graphs.

\subsection{Non-commutative graphs as operator systems}\label{Non-commutative graphs as operator systems}

\begin{definition}
Let $G=(V,E)$ be a graph with vertex set $[n].$
Define $S_G \subset M_n$ by
\begin{align*}
S_G:=\Span \lbrace E_{i,j}:(i,j) \in E \text{ or }i=j \rbrace.
\end{align*}
\end{definition}

Observe that for any graph $G$, $S_G$ will be a submatricial operator system. In \cite{PO}, it is shown that graphs $G$ and $H$ are isomorphic if and only if $S_G$ and $S_H$ are isomorphic in the category of operator systems. We discuss this in more details in \ref{complement sec}.

Given a graph $G$, if vertices $i ,j$ are not adjacent, then $e_ie_j^*=E_{i,j} $ is orthogonal to the submatricial operator system $S_G$. Similarly if $ \lbrace i_1, \dots, i_k \rbrace$ is an independent set of vertices in $G$ then for any $j \neq k$ we have $e_{i_{j}}e^*_{i_{k}} $ is orthogonal to $S_G$. If $v =(v_1,\dots, v_k)$ is an orthonormal collection of vectors in $\C^n$ then $v$ called an \emph{independent set} for a submatricial operator system $S \subset M_n$ if for any $i \neq j$, $v_iv_j^*$ is orthogonal to $S$.

\begin{definition}
Let $S$ be a submatricial operator system. We define the \emph{independence number}, $\alpha(S)$, to be the largest $k \in \N$ such that there exists an independent set for $S$ of size $k$.
\end{definition}

A graph $G=(V,E)$ has a $k$-colouring if and only if there exists a partition of $V$ into $k$ independent sets. In \cite{Lecture} Paulsen defines a natural generalization of the chromatic number to non-commutative graphs. We say a submatricial operator system $S \subset M_n$ has \emph{$k$-colouring} if there exists an orthonormal basis for $\C^n$, $v=(v_1, \dots , v_n)$, such that $v$ can be partitioned into $k$ independent sets for $S$.

\begin{definition}
Let $S \subset M_n$ be a submatricial operator system. The \emph{chromatic number}, $\chi(S)$, is the least $k \in \N$ such that $S$ has a $k$-colouring.
\end{definition}

For any submatricial operator system $S \subset M_n$ we have $\chi(S) \leq n$ since you can partition any basis of $\C^n$ into $n$ independent sets. In Theorem~\ref{thm: insensitive} we show that both of the above parameters provide a generalization of the classical graph theory parameters, that is we show $\alpha(S_G) = \alpha(G)$ and $\chi(S_G)= \chi(G)$.

\begin{example}
Consider the submatricial operator system $S:= \Span \lbrace I, E_{i,j} : i \neq j \rbrace \subset M_n$. Let $u_1 , u_2$ be two orthonormal vectors and let $i$ be an element of the support of $u_1$. Since $u^*_1u_2 =0 $ there must be an element $j \neq i$ of the support of $u_2$. Then  $\langle u_1u_2^*, E_{i,j} \rangle = u_1(i) \overline{u_2(j)} \neq 0 $. Thus we see that $ \alpha(S) =1$. This also tell us that $\chi(S) =n$.
\end{example}
As in \cite{DSW}, given a graph $G$ one can compute the Lov\'{a}sz theta number $\vth(G)$ as,
 \begin{align*}
\vth(G) = \max \lbrace \|I+T\| : I+T \geq 0, \, T_{i,j} = 0 \text{ for } i \sim j \rbrace.
\end{align*}
Here the supremum is taken over all $n \times n$ matrices and  $I+T \geq 0$ indicates that $I+T$ is positive semidefinite.

\begin{theorem}\label{Lsandwich}
Let $G$ be a graph and $\cl{G}$ be the graph complement of $G$. Then,
\begin{align*}
\alpha(G) \leq \vth(G) \leq \chi(\cl{G}).
\end{align*}
\end{theorem}

In order to obtain an generalization of \ref{Lsandwich} we need to identify the the appropriate generalization of a graph complement. Given a submatricial operator system $S \subset M_n$ we use the orthogonal complement $S^\perp$ to generalize the graph complement. Note that the orthogonal complement of a submatricial operator system is no longer a submatricial operator system since it will fail to contain the identity operator. In fact since $I \in S$ we will have $\tr(A) = \langle A, I \rangle =0$ for every $A$ element of $S^\perp$. In \cite{St}, Stahlke works with precisely these objects. We show that it is useful to consider both submatricial operator systems and submatricial traceless self-adjoint operator spaces to generalize the graph complement.

\subsection{The complement of a non-commutative graph}\label{complement sec}

In this section, we introduce the analogue of the notion of a graph complement for non-commutative graphs. Using this, we define a notion of clique number independence number and chromatic number.

\begin{definition}
Let $G$ be a finite graph with vertex set $[n]$. The \emph{traceless self-adjoint operator space associated to $G$} is the linear space
      \begin{align*}
        \calJ_G := \Span\{E_{i,j}: i \sim j \} \subset M_{n}\;.
      \end{align*}
A \emph{traceless non-commutative graph} is any submatricial traceless self-adjoint operator space.
\end{definition}

\begin{remark}
  The traceless self-adjoint operator space $\calJ_\calG$ is the traceless non-commutative graph $S_G$ given in \cite{St}. Given a finite graph $G$ with vertex set $[n]$, we have the identity $\calJ_{G}^\perp = S_{\cl{G}}$.
  This identity in particular suggests that the graph complement of a non-commutative graph should be its orthogonal complement. In \cite{St}, Stalhke suggests that the graph complement of  $\calJ_G$ should be $(\calJ_G + \C I)^{\perp}$. However, this notion of complement would mean that $\calJ_{\cl{G}} \neq (\calJ_G + \C I)^{\perp}$ for any graph with at least two vertices. We shall see that, so long as one is willing to pay the price of working with two different notions of a non-commutative graph, the orthogonal complement is the correct analogue of the graph complement.
\end{remark}

\begin{proposition}
  The traceless self-adjoint operator subspaces of $M_n$ are exactly the orthogonal complements of submatricial operator systems. That is, $S$ is a submatricial operator system if and only if $S^\perp$ is a traceless self-adjoint operator space.
\end{proposition}

\begin{proof}
  If $S$ is an operator subsystem of $M_n$ then for any $X \in S^\perp$, $tr(X) = \Angle{X,I} = 0$. As well, if $X \in S^\perp$, for any $Y \in S$, $\tr(XY) = \cl{\tr(Y^*X^*)} = 0$. This proves that $S^\perp$ is a traceless self-adjoint operator space. Conversely, if $S$ is a traceless self-adjoint operator space, then $S^\perp$ contains $I$ since for all $X \in S$, $\Angle{X,I} = \tr(X^*I) = 0$. If $X \in S^\perp$ then $\Angle{X^*,Y} = \tr(XY) = \cl{\tr(Y^*X^*)} = 0$. Therefore, $X^* \in S^\perp$. This proves that $S^\perp$ is an operator system.
\end{proof}

\begin{proposition}
  If $G$ is a graph with vertex set $[n]$ then $S_G^\perp = \calJ_{\cl{G}}$.
\end{proposition}

\begin{proof}
  Observe that for $i,j,k,l \in [n]$, $E_{ij} \in S_G^\perp$ if and only if for all $k \simeq_G l$, $\tr(E_{ij}E_{kl}) = 0$. This is only possible if $i \sim_{\cl{G}} j$.
\end{proof}

It is a result of Paulsen and Ortiz~\cite[Proposition 3.1]{PO} that two graphs $G$ and $H$ of the same vertex set $[n]$ are isomorphic if and only if there is a $n \times n$ unitary matrix $U$ for which $U S_G U^* = S_H$.

\begin{corollary}
  Suppose that $G$ and $H$ are graphs with vertex set $[n]$. The graphs $G$ and $H$ are isomorphic if and only if there is an $n \times n$ unitary matrix $U$ such that $U \calJ_G U^* = \calJ_H$.
\end{corollary}

\begin{proof}
  For any $n \times n$ unitary matrix $U$, $(U S_G U^*)^{\perp} = U \calJ_{\cl{G}} U^*$. Since $G$ and $H$ are isomorphic if and only if their graph complement is, the result follows.
\end{proof}

\begin{remark}
In \cite{Weaver} a \emph{qunatum graph} is defined as a reflexive, symmetric quantum relation on a $*$-subalgebra $\calM \subseteq M_n$. In this framework a submatricial operator system $S$ is indeed quantum graph when taking $\calM = M_n$. This approach fails to provide a complement for a quantum graph since $S^\perp$ will fail to be a reflexive quantum relation on any $\calM \subseteq M_n$ and hence will not be a quantum graph.
\end{remark}

 The notion of an independence set for an submatricial operator system was described solely in terms of an orthogonality relation. We can similarly say that an orthonormal collection of vectors $v=(v_1, \dots, v_n)$ in $\C^n$ is an \emph{independent set} for a submatricial traceless self-adjoint operator space $\calJ$ $\subset M_n$ if for any $i \neq j$, $v_iv_j^*$ is orthogonal to $\calJ$. We say $\calJ$ has a $k$-coloring if there exists an orthonormal basis $v=(v_1, \dots, v_n)$ of $\C^n$, that can be partitioned into $k$ independent sets for $\calJ$.

\begin{definition} Let $\calJ \subset M_n$ be a submatricial traceless self-adjoint operator space.
\begin{enumerate}
\item The independence number, $\alpha(\cal{J})$, is the largest $k \in \N$ such that there exists an independent set of size $k$ for $\calJ$.

\item The chromatic number $\chi(\calJ)$ is the least integer $k$ such that $\calJ$ has $k$-colouring.
\end{enumerate}
\end{definition}

It is not hard to show that $\chi$ is monotonic and $\alpha$ is reverse monotonic under inclusion. This holds when considering these as parameters on submatricial operator systems as well as submatricial traceless self-adjoint operator spaces.

 Next we show that if $G$ is a graph, $S_G$ and $\calJ_G$ have the same independence number and chromatic number.
We start with a lemma. The following proof is in \cite[Lemma 7.28]{Lecture}:
\begin{lemma}\label{lemma: Vern, 7.28}
  Let $v_1,\ldots, v_n$ be a basis for $\C^n$. There exists a permutation $\sigma$ on $[n]$ so that for each $i$, the $\sigma(i)$th component of $v_i$ is non-zero.
\end{lemma}

\begin{proof}
  Let $A = [a_{i,j}]$ denote the matrix with column $i$ equal to $v_i$. Since we have a basis, $\det(A) \neq 0$. But
  \begin{align*}
    \det(A) = \sum_{\sigma \in Sym([n])} sgn(\sigma)a_{1,\sigma(1)}\cdots a_{n,\sigma(n)}\;.
  \end{align*}
  There must therefore be some $\sigma$ for which the product $a_{1,\sigma(1)}\cdots a_{n,\sigma(n)}$ is non-zero. This permutation works.
\end{proof}

\begin{theorem}\label{thm: insensitive}
Let $G$ be a graph on $n$ vertices, we have $\alpha(G) = \alpha(S_{G}) = \alpha(\calJ_{G})$ and $\chi(G) = \chi(S_G) = \chi(\calJ_G)$.
\end{theorem}

\begin{proof}
The inclusion $\alpha(S_G) \leq \alpha(\calJ_G)$ follow from reverse monotonicity. If $i_1 , \dots, i_k$ are an independent set of vertices in the graph $G$ then we have that the standard vectors $e_{i_{1}}, \dots, e_{i_{k}}$ is an independent set for $S_G$ so we get
\begin{align*}
  \alpha(G) \leq \alpha(S_G) \leq \alpha(\calJ_G)\;.
\end{align*}
Next suppose that $v_1, \dots ,v_k$ are an independent set for $\calJ_G$. Then since $v_1 , \dots ,v_k$ is an linearly independent set of vectors we can find a permutaiton $\sigma$ on $[n]$ so that $\Angle{v_i, e_{\sigma(i)}}$ is non-zero for all $i$.

We note that if vertices $\sigma(j)$ and $\sigma(k)$ are adjacent in $G$ then we have $E_{\sigma(j),\sigma(k)} \in \calJ_G$. But then $\langle v_jv^*_k, E_{\sigma(j),\sigma(k)} \rangle = \Angle{v_j,e_\sigma(j)} \Angle{e_{\sigma(k)},v_k} \neq 0 $ a contradiction. Thus $\sigma(1),\ldots, \sigma(k)$ are an independent set for the graph $G$ so $\alpha(\calJ_G) \leq \alpha(G)$. The proof for $\chi$ follows the same argument.
\end{proof}

Recall for a classical graph $G$ the clique number, $\omega(G)$, satisfies that $\omega(G) = \alpha(\cl{G})$.

\begin{definition} Let $S$ be a submatricial operator system and let $\calJ$ be a submatricial traceless self-adjoint operator space.

\begin{enumerate}
\item Define the clique number, $\omega(S)$, to be the independence number of the submatricial traceless self-adjoint operator space $S^\perp$.
\item Define the clique number, $\omega(\calJ)$, to be the independence number of the submatricial operator system $\calJ^\perp$.
\end{enumerate}
\end{definition}

We can use Theorem~\ref{thm: insensitive} to conclude that for any graph $G$ we have $\omega(G)= \omega(S_G) = \omega(\calJ_G)$.

The next proposition shows that $\alpha$, $\omega$, and $\chi$ may be computed purely from the associated parameters for graphs. We can achieve this by associating a family of graphs to each submatricial traceless self-adjoint operator space or submatricial operator system.

\begin{definition}\label{def: confusability graph of a basis}
Given a submatricial operator system $S \subseteq M_n$ and an orthonormal basis $v = (v_1, \dots , v_n)$ we can construct two different graphs.

\begin{enumerate}

\item \emph{The confusability graph} of $v$, with respect to $S$, denoted $H_{v}(S) $, is the graph on $n$ vertices with $i \sim j$ if and only if $v_{i}v_j^* \in S$.

\item \emph{The distinguishability graph} of $v$, with respect to $S$, denoted $G_v(S)$ is the graph on $n$ vertices with $i \sim j$ if and only if $v_iv_j^* \perp S$.
\end{enumerate}
\end{definition}

We can also define the confusability and distinguishability graphs of an orthonormal basis $v= (v_1, \dots ,v_n)$ with respect to submatricial traceless self-adjoint operator spaces $\calJ \subseteq M_n$ in the same way. We would then have for $S = \calJ^\perp$,  $G_{v}(S) = H_v(\calJ)$ and $H_v(S) = G_v(\calJ).$ When it is clear what the underlying system or submatricial traceless self-adjoint operator space is we simply write $G_v$ and $H_v$.

\begin{theorem}\label{thm: Non-commuatative parameters from graph paramaters}
 Let $\calJ$ be a submatricial traceless self-adjoint operator space in $M_n$ and let $\calB$ denote the set of ordered orthonormal bases for $\C^n$. We have the identities
  \begin{align*}
    \alpha(\calJ) &= \sup_{v \in \calB} \alpha(\cl{G_v})\;, \\ \chi(\calJ) &= \inf_{v \in \calB} \chi(\cl{G_v})\;, \text{ and } \\ \omega(\calJ) &= \sup_{v \in \calB} \omega (H_v)\;.
  \end{align*}
  The same identity holds if we replace $\calJ$ with a submatricial operator system in $M_n$.
\end{theorem}

\begin{proof}
Suppose $v_1 , \dots, v_c$ is a maximal independent set for $\calJ$, that is for $i \neq j$ we have $v_iv_j^* \perp \calJ$. We can extend this collection to an orthonormal basis $v= (v_1, \dots, v_c, v_{c+1}, \dots v_n)$. Note that the vertices $1, \dots ,c$ in the graph $\cl{G_v}$ are an independence set since for distinct $i ,j \in [c] $ we have $v_iv_j^* \perp \calJ $. This gives $i \sim j$ in $G_v$. Thus there is no edge between $i$ and $j$ in $\cl{G_v}$. Therefore $\alpha(\cl{G_v}) \geq c$ so we have $\alpha(\calJ) \leq \sup_{v \in \calB} \alpha(\cl{G_v})$.
Conversely, for each $v \in \calB$ if $i_1, \dots, i_c$ are an independent set for $\cl{G_v}$ then $i_{j} \sim i_{k}$ in $G_V$. We then have $v_{i_{1}} ,\dots, v_{i_{c}}$ is an independent set for $\calJ$. This gives $\alpha(\calJ) \geq \alpha(\cl{G_v})$.

The proof of the second identity is similar. If $\chi(\calJ) =c$ then there exists orthonormal basis $v=(v_1, \dots ,v_n)$ and a partition $P_1 , \dots P_c$ of $[n]$ such that $v_iv_j^* \perp \calJ$ for distinct $i$ and $j$ in the same partition. Define a colouring $f$ of $\cl{G_v}$ by having $f(i)= l$ if and only if $i \in P_l$. We see that for $i \neq j$ if we have $f(i) =f(j)$ then $v_iv_j^* \perp \calJ$ giving that $i \sim j$ in $G_v$ so $f$ is indeed a $c$ colouring of $\cl{G_v}$. This gives $\chi(\calJ) \geq  \inf_{v \in \calB} \chi(\cl{G_v})$.
Conversely, if $f$ is any c colouring of $\cl{G_v}$ for some $v \in \calB$ then we can obtain a $c$ colouring of $\calJ$ by partitioning $[n]$ into sets $P_1, \dots P_c$ where $i \in P_l$ if and only if $f(i)=l$. Then if distinct $i,j \in P_l$ we have $i \sim j $ in $G_v$ so $v_iv_j^* \perp \calJ$.

Lastly, suppose $(v_1, \dots, v_k) $ is a collection of orthonormal vectors such that for distinc $i,j$ we have $v_iv_j^* \in \calJ$. We can extend this set to a orthonormal basis $v=(v_1, \dots v_n)$ and we immediately get that the vertices $\lbrace 1 , \dots , k \rbrace$ form a clique in $H_v$. For the other direction note for any basis $v=(v_1, \dots v_n)$ $H_v$ has a clique $i_1, \dots ,i_k$ then $v_1, \dots v_k$ will satisfy $v_iv_j^* \in \calJ$ for distinct $i$ and $j$.
\end{proof}

We now extend the definition of Lov\'{a}sz' theta function to non-commutative graphs.

\begin{definition}
Let $S$ be a submatricial operator system and $\calJ$ be a submatricial traceless self-adjoint operator space. Define the theta number of a submatricial operator system, $\vth(S)$ and the complementary theta number of a submatricial traceless self-adjoint operator space, $\cl{\vth}$ as follows.
\begin{enumerate}
\item $ \vth(S) = \sup \lbrace \|I +T \| : T \in M_n, \, I+T \geq 0, \, T \perp S \rbrace.$

\item $ \cl{\vth}(\calJ) = \sup \lbrace \|I+T \| : T \in M_n, \, I+T \geq 0, T \in \calJ \rbrace $.
\end{enumerate}
\end{definition}

Observe that $\vth(S_G) = \vth(G)$ and $\cl{\vth}(\calJ_G) = \vth( \cl{G})$ for all graphs $G$.

\begin{example}
Recall the previously mentioned submatricial operator system $S:= \Span \lbrace I, E_{i,j} : i \neq j \rbrace \subset M_n$. We see that $\vth(S) =n$ since we can take $T$ to be the diagonal matrix with $n-1$ for the $1,1$ entry and $-1$ for all other diagonal entries. We also see that if $v=(e_1, \dots, e_n)$ is the standard basis for $C^n$ then $v$ is a clqiue for $S$ and thus we have $\chi(S^\perp)=1$. This shows that using the definition of the chromatic number from \cite{Lecture} we can not hope to generalize the Lov\'{a}sz sandwich theorem.
\end{example}

Motivated by applications in quantum information theory many non-commutative graph parameters can have \emph{entanglement assisted} and \emph{completely bounded} versions. We define them here for later use. Let $X$ be a submatricial operator system or a submatricial traceless self-adjoint operator space and let $M_d(X)$ denote $d$ by $d$ matrices with entries from $X$.. For $d \geq 1$, the \emph{$d$th degree independence number} of $X$, denoted $\alpha_d(X)$, is defined by $\alpha_d(X)= \alpha(M_d(X))$. We then define \emph{the completely bounded independence number}, $\alpha_{cb}(X)$, as the supremum over $d \geq 1$ of $\alpha_d(X)$. Similarly for a submatricial operator system $S$ we define the \emph{$d$th degree theta number} $\vth_d(S)$ to be $\vth(M_d(S))$ and let $\theta_{cb}(S) = \inf \theta_{d}(S)$ denote the \emph{the completely bounded theta number} of $S$. For a submatricial traceless self-adjoint operator space $\calJ$ we let $\cl{\vth}_d(\calJ) = \cl{\vth}(M_d(\calJ))$ and use $\cl{\vth}_{cb}(\calJ)$ to denote the supremum over $d \geq 1$ of $\cl{\vth}_d(\calJ)$.

\section{Non-commutative Lov\'{a}sz inequality}\label{Sand section}
We see by the previous example that one needs a different generalization of the chromatic number in order to obtain a Lov\'{a}sz sandwich Theorem for non-commutative graphs. Here we introduce the strong and minimal chromatic number of a submatricial operator system and provide a generalization on Lov\'{a}sz theorem.

\subsection{The Strong chromatic number}

Let $\calJ \subset M_n$ be a submatricial traceless self-adjoint operator space. A collection of orthonormal vectors $v=(v_1, \dots ,v_k)$ in $\C^n$ is called a \emph{strong independent set for $\calJ$} if for any $i, j$,we have $v_iv_j^*$ is orthogonal to $\calJ$. We say that $\calJ$ has a \emph{strong $k$-colouring} if there exist an orthonormal basis $v=(v_1, \dots, v_n)$ of $\C^n$ that can be partitioned into $k$ strong independent sets for $\calJ$. We will show in Corollary \ref{cor: hatchi is faithful}, $\Hat{\chi}(\calJ_G)$ agrees with the chromatic number of $G$, for any graph $G$.

\begin{definition}
Let $\calJ \subset M_n$ be a submatricial traceless self-adjoint operator space. The \emph{strong chromatic number}, $\widehat{\chi}(\calJ)$, is the least $k \in \N$ such that $\calJ$ has a strong $k$-colouring. If $\calJ$ has no strong-$k$ colouring then we say $\widehat{\chi}(\calJ)= \infty$.
\end{definition}

As with $\chi$ we have $\widehat{\chi}$ is monotonic with respect to inclusion.

\begin{example}\label{Example: Traceless diagonals }
  Suppose that $S = \C1+ \Span\{E_{i,j}: i \neq j\} \subset M_n$. Define $v_k = (1,\zeta^k, \zeta^{2k},\ldots, \zeta^{(n-1)k})$. Observe that the $v_i$ are orthogonal and that $v_kv_k^*$ belongs to $S$ for all $k$. Thus $S^\perp$ does have a strong-$n$ colouring and we get $\widehat{\chi}(S^\perp) \leq n$.
\end{example}

\begin{example}
  Consider the submatricial traceless self-adjoint operator space $\calJ = \C \Delta \subset M_n$ where $\Delta = \diag(n, -1,-1,-1,\ldots, -1)$. Observe that $\calJ \subset S^\perp$. By monotonicity, $\widehat{\chi}(\calJ) \leq \widehat{\chi}(S^\perp) \leq n$. It is known that $\cl{\vth}(\calJ) = n$ (see \cite[Remark 4.3]{PO}). We show in Theorem \ref{Generalized Lovasz Inequality} that $\widehat{\chi}$ is bounded below by $\cl{\vth}$ Thus we have $\widehat{\chi}(\calJ)= \widehat{\chi}(S^\perp) =n$.
\end{example}
In \cite{St} Stahlke introduces a different chromatic number for submatricial traceless self-adjoint operator spaces.

\begin{definition}
  Let $\calJ$ and $\calK$ be submatricial traceless self-adjoint operator spaces in $M_n$ and $M_m$ respectively. We say that there is a \emph{graph homomorphism from $\calJ$ to $\calK$}, denoted $\calJ \to \calK$, if there is a cptp (completely positive and trace preserving) map $\calE : M_n \to M_m$ with associated Kraus operators $E_1,\ldots, E_r$ for which $E_i \calJ E^*_j \subset \calK$ for any $i$ and $j$.
\end{definition}
Stalhke's chromatic number of a submatricial traceless self-adjoint operator space $\calJ$, denoted $\chi_{St}(\calJ)$, is the least integer $c$ for which there is a graph homomorphism $\calJ \to \calJ_{K_c}$ if one exists. We set $\chi_{St}(\calJ)=\infty$ otherwise.

Observe that $\chi_{St}$ is monotonic under graph homomorphism by construction.

\begin{theorem}\label{thm: clchi <chi_St}
For any submatricial traceless self-adjoint operator space $\calJ \subset M_n$ we have $\widehat{\chi} (\calJ) \geq \chi_{St}(\calJ)$.
\end{theorem}

\begin{proof}
Suppose $\widehat{\chi}(\calJ) = r$. There exists a orthonormal basis $v_1, \dots ,v_n$ that can be partitioned into strong independent sets $P_1 , \dots ,P_r$. By reordering the vectors, we may assume that whenever $v_i \in P_\ell$ and $v_j \in P_{\ell +1}$, that $i < j$. By conjugating by the unitary $U: v_i \mapsto e_i$, we get the inclusion $\bigoplus^{r}_{i=1} M_{|P_{i}|} \subset U \calJ^\perp U^* = (U\calJ U^*)^\perp$.

 This then gives us  $(\bigoplus^{r}_{i=1} M_{|P_{i}|})^\perp \supset (U\calJ U^*)$. We have that $\calJ \rightarrow U\calJ U^*$ by conjugating by the unitary $U$. Similarly we have $U \calJ U^* \rightarrow (\oplus^{r}_{i=1} M_{|P_{i}|})^\perp $ by inclusion. Since $\chi_{St}$ is monotonic with respect to homomorphisms we get $ \chi_{St}(\calJ) \leq \chi_{St}((\oplus^{r}_{i=1} M_{|P_{i}|})^\perp) = \chi(\cl{G}) = r$ where $G$ is the disjoint union of $r$ complete graphs.
\end{proof}

\begin{corollary}
  If $\calJ \subset M_n$ is a submatricial traceless self-adjoint operator space for which for some basis $v=(v_1,\ldots, v_n)$, the diagonals $v_iv_i^*$ are orthogonal to $\calJ$, then $\chi_{St}(\calJ) \leq n$.
\end{corollary}

We can also define the completely bounded version of the strong chromatic number. For $d \geq 1$ let $\widehat{\chi}_d(\calJ) = \widehat{\chi}(M_d(\calJ))$ for $\calJ$ a submatricial traceless self-adjoint operator space. Then $\widehat{\chi}_{cb}$ is defined as the infimum over $d \geq 1$ of $\widehat{\chi}_d(\calJ)$. Note that if $\calJ \subset M_n$ and $v=(v_1, \dots,v_n)$ can be partitioned into strong independent sets $P_1, \dots, P_k$. Let $Q_{s} = \lbrace v_i \otimes e_j : v_i \in P_s \text{ and } j=1, \dots, d \rbrace $. Then $Q_1, \dots , Q_k$ partition the orthonormal basis $\lbrace v_i \otimes e_j : i=1, \dots, n, \, j=1, \dots, d \rbrace$ into $k$ strong independent sets for $M_d(\calJ)$. Thus we have  $\widehat{\chi}_d(\calJ) \leq \widehat{\chi}(\calJ)$ for all $d \geq 1$.

Recall that for $d,n \geq 1$, the partial trace map is
\begin{align*}
  M_d \otimes M_n \to M_n: X \otimes Y \mapsto \tr(X) Y\;.
\end{align*}

\begin{corollary}
  For any submatricial traceless self-adjoint operator space $\calJ \subset M_n$ we have $\widehat{\chi}_{cb} (\calJ)\geq \chi_{St}(\calJ)$.
\end{corollary}

\begin{proof}
  Observe that $M_n \to M_d \otimes M_n :X \mapsto \frac{1}{d} 1 \otimes X$ is a graph homomorphism as is the partial trace $M_d \otimes M_n \to M_n$. It then suffices to check that $\widehat{\chi}(M_d(\calJ)) \geq \chi_{St}(M_d(\calJ))$. applying Theorem~\ref{thm: clchi <chi_St} get us the result.
\end{proof}

As is the case with $\chi$ we can approximate $\widehat{\chi}$ using the chromatic number for classical graphs.

\begin{theorem}\label{thm: chromatic to graph}
  Let $\calJ \subset M_n$ be a submatricial traceless self-adjoint operator space. Suppose that $\calB_\calJ$ denotes the set of ordered orthonormal bases $v = (v_1,\ldots, v_n)$ of $\C^n$ for which $v_iv_i^* \in \calJ$ for all $i$. For each $ v = (v_1,\ldots, v_n)$ in $\calB_\calJ$, define the graph $G_v$ with vertices $[n]$ and edge relation given by $i \sim j$ if $v_iv_j^*$ is orthogonal to  $\calJ$. Then,
  \begin{align*}
    \widehat{\chi}(\calJ) = \inf_{v \in \calB} \chi(\cl{G_v})\;,
  \end{align*}
  whenever $\widehat{\chi}(\calJ)$ is finite.
  
\end{theorem}

\begin{proof}
  The proof is exactly as in Theorem~\ref{thm: Non-commuatative parameters from graph paramaters}.
\end{proof}

\begin{corollary}\label{cor: hatchi is faithful}
  For any finite graph $G$, $\Hat{\chi}(\calJ_G) = \chi(G)$.
\end{corollary}

\begin{proof}
  By Theorem~\ref{thm: chromatic to graph}, $\Hat{\chi}(\calJ_G) \leq \chi(\cl{G_v})$ where $v= (e_1,\ldots, e_n)$. The complement of the graph $G_v$ is the graph $G$. This gets us the bound $\Hat{\chi}(\calJ_G) \leq \chi(G)$. As well, by Theorem~\ref{thm: clchi <chi_St}, $\chi(G) = \chi_{St}(\calJ_G) \leq \Hat{\chi}(\calJ_G)$.
\end{proof}

Using the strong chromatic number we are easily able to generalize other graph inequalities that for now remain unanswered for $\chi_{St}$. In \cite{St} Stahlke asks if one can show $\chi_{St}(\calJ) \omega(\cl{\calJ}) \geq n$ for all submatricial traceless self-adjoint operator spaces $\calJ \subset M_n$. The question is motivated by the simple graph inequality $\chi(G)\omega(\cl{G}) \geq n$. Indeed for $\calJ \subset M_n$ a submatricial traceless self-adjoint operator space if we suppose $\widehat{\chi}(\calJ) =k$ then we can find an orthonormal basis $v=(v_1, \dots, v_n)$ and a partition of $v$ into independent sets $P_1, \dots, P_k$. By definition of $\omega(\calJ^\perp)$ we know that $|P_i| \leq \omega(\calJ^\perp)$ for $i =1, \dots, k$. Thus we have $n = \sum_{i} |P_i| \leq \sum_{i} \omega(\calJ^\perp) = \widehat{\chi}(\calJ)\omega(\calJ^\perp)$.

Using \cite{St}, one can establish that $\cl{\vth}(\calJ) \leq \chi_{St}(\calJ)$ for any submatricial traceless self-adjoint operator space $\calJ\subset M_n$: if $c= \chi_{St}(\calJ)$, then there is a graph homomorphism $\calJ \to \calJ_{K_c}$. In \cite[Theorem 19]{St}, they establish that $\cl{\vth}_n$ is monotonic under graph homomorphisms. We therefore get the inequality
\begin{align*}
  \cl{\vth}_n(\calJ) \leq \cl{\vth}_n(\calJ_{K_c}) = \vth(\cl{K_c}) \leq \chi(K_c) = c\;.
\end{align*}

We can now establish a Lov\'{a}sz sandwich Theorem for $\widehat{\chi}$.

\begin{theorem}\label{Generalized Lovasz Inequality}
  Let $S$ be a submatricial operator system. For any $d \geq 1$, we have the inequalities
  \begin{align*}
    \alpha_d(S) \leq \vth_d(S) \leq \widehat{\chi}_d(S^\perp)\;.
  \end{align*}
\end{theorem}

\begin{proof}
  It suffices to check for $d = 1$. The inequality $\alpha(S) \leq \vth(S)$ is a result in \cite[Lemma 7]{DSW} so we will only prove the other inequality. Let $v = (v_1,\ldots, v_n)$ be an orthonormal basis that can be partitioned into $k$ strong independent sets for $S^\perp$. Then consider the graph $G_v(S^\perp) $ as defined in Theorem~\ref{thm: Non-commuatative parameters from graph paramaters}. We have $\widehat{\chi}(S^\perp) = \chi(\cl{G_v})$. There exists a unitary $U \in M_n$ such that we get the the inclusion $S \supset US_{G_v}U^*$. Since $\vth$ is reverse monotonic under inclusion and invariant under conjugation by a unitary, we establish the inequalities
  \begin{align*}
    \vth(S) \leq \vth(S_{G_v}) = \vth(G_v) \leq \chi(\cl{G_v})= \widehat{\chi}(S^\perp).
  \end{align*}
\end{proof}

Similarly for any $d \geq 1$ we get the follow inequality for any submatricial traceless self-adjoint operator space $\calJ$.
\begin{align*}
\alpha_d(\calJ^\perp) \leq \cl{\vth}_d(\calJ) \leq \widehat{\chi}_d(\calJ).
\end{align*}

\subsection{The minimal chromatic number}
In this section, we wish to construct a concrete definition of a homomorphism monotone chromatic number.

\begin{definition}
  Let $\calJ \subset M_n$ be a submatricial traceless self-adjoint operator space. Define the minimal chromatic number of $\calJ$, denoted $\chi_0(\calJ)$, to be the least integer $c$ for which there exists a basis $v_1,\ldots, v_n$ of $\C^n$ and a partition $P_1,\ldots, P_c$ of $[n]$ for which whenever $i,j \in P_s$, we have the relation $v_iv_j^* \perp \calJ$.
\end{definition}

This parameter also agrees with the chromatic number for graphs.

\begin{proposition}
  Let $G$ be a finite graph. We have the relation $\chi(G) = \chi_0(\calJ_G)$.
\end{proposition}

\begin{proof}
  Since $\chi_0(\calJ_G) \leq \chi(G)$, it suffices to show that $\chi(G) \leq \chi_0(\calJ_G)$. For this proof, let $c$ be minimal and let $v_1,\ldots, v_n$ be a basis in $\C^n$ for which there is a partition $P_1,\ldots, P_c$ of $[n]$ such that whenever $i,j$ in $P_s$, $v_iv_j^* \perp \calJ_G$. We then have a permutation $\sigma$ of $[n]$ for which $\Angle{v_i, e_{\sigma(i)}}$ is non-zero. By conjugating $\calJ_G$ by the permutation matrix defined by $\sigma$, assume that $\sigma(i) = i$ for all $i$. Define the $c$-colouring $f: V(G) \to [c]$ By $f(i) = s$ for $s$ such that $i \in P_s$. To see that this is a colouring, suppose not. There are then $i \sim j$ for which $i,j \in P_s$ for some $s$. By definition then we have, $E_{i,j}$ belongs to $\calJ_G$. We observe then, 
  \begin{align*}
    \Angle{v_iv_j^*, E_{i,j}} = \tr(v_jv_i^*e_ie_j^*) = \Angle{v_i,e_i}\Angle{v_j,e_j} \neq 0\;.
  \end{align*}
  This is contradicts the fact that $v_iv_j^* \in \calJ^\perp$.
\end{proof}

We recall the following result, which arises as a consequence of the Stinespring dilation Theorem (see \cite[Definition 7]{St}).

\begin{lemma}\label{lemma: graph homomorphism}
  Let $\calJ \subset M_n$ and $\calK \subset M_m$ be submatricial traceless self-adjoint operator spaces. There is a graph homomorphism $\calJ \to \calK$ if and only if there is a $d \geq 1$ and an isometry $E : \C^{n} \to \C^{m} \otimes \C^d$ for which $E \calJ E^* \subset M_d(\calK)$.
\end{lemma}

We use this equivalent characterization to show that $\chi_{0,cb}$ is monotonic under graph homomorphisms.

\begin{theorem}\label{theorem: chi_0 monotone}
  Let $\calJ \subset M_n$ and $\calK \subset M_m$ be submatricial traceless self-adjoint operator spaces. If there is a graph homomorphism $\phi: \calJ \to \calK$ with $d$ associated Kraus operators, then we have the inequality
  \begin{align*}
    \chi_0(\calJ) \leq \chi_0(M_d \otimes \calK)\;.
  \end{align*}
  In particular, $\chi_{0,cb}(\calJ) \leq \chi_{0,cb}(\calK)$.
\end{theorem}

\begin{proof}
  Suppose that $P_1,\ldots, P_c$ is a partition of the set $[d \times m]$ and $(w_{i}: i \in [d \times m])$ is a basis for which whenever $i,j$ are in the same partition $P_s$, then $w_iw_j^* \perp M_d \otimes \calK$. By lemma~\ref{lemma: graph homomorphism}, there is an isometry $E$ for which the map $\phi: M_n \to M_d \otimes M_m : X \mapsto EXE^*$ sends $\calJ$ to $M_d \otimes \calK$. Consider the set $\{E^*w_i : i \in [d \times m]\}$. This set spans $\C^n$. To see this, for any $v \in \C^n$, since $Ev \in \C^d \otimes \C^m$, there are some $\lambda_i$ for which $Ev = \sum_i \lambda_i w_i$. Multiplying on the left by $E^*$ tell us that $v$ is spanned by the $E^*w_i$. If $i,j$ belong to the same partition $P_s$, then for any $X \in \calJ$,
  \begin{align*}
    \Angle{E^*w_i(E^*w_j)^*, X} = \Angle{w_iw_j^*, EXE^*} = 0\;.
  \end{align*}

  For each $i \in [c]$, let $C_i = \{E^*w_j : j \in P_i \}$. We will define a sequence of linearly independent subspaces $V_1,\ldots, V_c$ for which $\sum_{i=1}^c V_i = \C^n$ inductively. For the base case, set $V_1 = \Span C_1$. For $i > 1$, let
  \begin{align*}
     V_i = \Span \left\{v \in \Span C_i:  v \not\in \sum_{k < i} V_k \right\}\;.
  \end{align*}
  By construction, the $V_i$ are linearly independent and $\sum_i V_i = \C^n$. For each $s$, let $Q_s = \{v_{s,1},\ldots, v_{s,d_s}\}$ be a basis in $V_s$, where $d_s = dim(V_s)$. Since these are a linear combination of the $C_s$, we get that whenever, $i,j \in [d_s]$, given any $X \in \calJ$,
  \begin{align*}
      \Angle{v_{s,i}v_{s,j}^*, X} = 0\;.
  \end{align*}
  The vectors $\{v_{s,i}: s \in [c], i \in [d_s]\}$ then form a basis for $\C^n$ and are partitioned by the sets $\{Q_s: s \in [c]\}$. This proves that $\chi_0(\calJ) \leq \chi_0(M_d \otimes \calK)$. If $r \geq 1$ and $E$ is an isometry for which the map $\phi: M_n \to M_d \otimes M_m : X \mapsto EXE^*$ sends $\calJ$ to $M_d(\calK)$, then the map
  \begin{align*}
      1 \otimes \phi: M_r \otimes M_n \to M_{r +d}\otimes M_m : X \otimes Y \mapsto X \otimes \phi(Y)
  \end{align*}
  is a map implemented by conjugation by the isometry $1 \otimes E$. By lemma~\ref{lemma: graph homomorphism}, $1 \otimes E$ is a graph homomorphism $M_r(\calJ) \to M_r(\calK)$. By the above proof, we get the bound $\chi_0(M_r(\calJ)) \leq \chi_0(M_{r+d}(\calK)) \leq \chi_{0,cb}(\calK)$ for every $r \geq 1$. This establishes the inequality
  \begin{align*}
    \chi_{0,cb}(\calJ) \leq \chi_{0,cb}(\calK)\;.
  \end{align*}
\end{proof}

  \begin{corollary}\label{cor: chi_0 < chi_St}
    Let $\calJ$ be a submatricial traceless self-adjoint operator space. We have the inequality
    \begin{align*}
      \chi_{0,cb}(\calJ) \leq \chi_{St}(\calJ)\;.
    \end{align*}
  \end{corollary}

  \begin{proof}
    We first show that $\chi_{0,d}(\calJ_G) = \chi(\calJ_G)$ for any $d \geq 1$ and any graph $G$. Let $G^{[d]}$ denote the graph on vertices $V(G) \times [d]$ for which $(v,i) \sim (w,j)$ if $v \sim w$ in $G$. The projection $G^{[d]} \to G: (v,i) \mapsto v$ and the inclusion $G \to G^{[d]} : v \mapsto (v,1)$ are graph homomorphisms. We therefore get by monotonicity of $\chi$ that $\chi(G) = \chi(G^{[d]})$. On the other hand, we know that $\chi_0(\calJ_G) = \chi(G) = \chi(G^{[d]}) = \chi_0(M_d(\calJ_G)) = \chi_{0,d}(\calJ_G)$. In particular, for any $c \geq 1$, $\chi_{0,cb}(\calJ_{K_c}) = \chi(K_c) = c$.
  \end{proof}

\section{Sabidussi's Theorem and Hedetniemi's conjecture } \label{Sabidussi and Hedet section}

 As an application of our new graph parameters, in this section, we generalize two results for chromatic numbers on graph products. For convenience we will let $\cl{\chi}(X)= \widehat{\chi}(X^\perp)$ for $X$ a submatricial traceless self-adjoint operator space or a submatricial operator system.
 \begin{definition}
   Let $G$ and $H$ be finite graphs.
   \begin{enumerate}
   \item Define the categorical product of $G$ and $H$ to be the graph $G \times H$ with vertex set $V(G) \times V(H)$ and edge relation given by $(v,a) \sim (w,b)$ if $v \sim_G w$ and $a \sim_H b$.

   \item Define the Cartesian product of $G$ and $H$ to be the graph $G \Box H$ with vertex set $V(G) \times V(H)$ and edge relation given by $(v,a) \sim (w,b)$ if one of the following holds
   \begin{enumerate}
    \item $v \sim_G w$ and $a =b$ or
    \item $v = w$ and $a \sim_H b$.
  \end{enumerate}

 \end{enumerate}
 \end{definition}

\subsection{Sabidussi's theorem}

We generalize the Theorem of Sabidussi.
\begin{theorem}[Sabidussi]
  Let $G$ and $H$ be finite graphs. We have the identity
  \begin{align*}
    \chi(G \Box H) = \max\{\chi(G), \chi(H)\}\;.
  \end{align*}
\end{theorem}

The first step in generalizing this Theorem is to generalize the cartesian product.

\begin{definition}
  Let $\calJ \subset M_n$ and let $\calK \subset M_m$ be submatricial traceless self-adjoint operator spaces. Let $v \subset \C^n$ and $w \subset \C^m$ be bases. Define the cartesian produt of $\calJ$ and $\calK$ relative to $(v,w)$ as the submatricial traceless self-adjoint operator space
  \begin{align*}
    (\calJ \Box \calK)_{v,w} = \calJ \otimes \calD_{w} + \calD_{v} \otimes \calK
  \end{align*}
  where for a basis $x= (x_1,\ldots, x_n)$, $\calD_{x} = \Span\{x_ix_i^*: i \in [n] \}$.

  In the case when $e = (e_1,\ldots, e_n)$ and $f = (e_1,\ldots, e_m)$, we define the cartesian product $\calJ \Box \calK$ to be $(\calJ \Box \calK)_{e,f}$.
\end{definition}

\begin{lemma}
  Let $G$ and $H$ be finite graphs with $[n] = V(G)$ and $[m] = V(H)$. We have the identity $\calJ_G \Box \calJ_H = \calJ_{G \Box H}$.
\end{lemma}

\begin{proof}
  Observe that $\calJ_G \otimes \calD_m = \Span\{E_{v,w} \otimes E_{i,i}: v \sim_G w, i \in [m] \}$ and that $\calD_n \otimes \calJ_H = \Span\{E_{i,i} \otimes E_{v,w}: i \in [n], v \sim_H w \}$. Combining these, we get that $E_{i,j} \otimes E_{k,l} \in \calJ_G \Box \calJ_H$ if and only if $i \sim_G j$ and $k = l$ or $i = j$ and $k \sim_H l$. This is exactly what it means to be a member of $\calJ_{G \Box H}$.
\end{proof}

\begin{lemma}\label{lemma: Box homomorphism}
  Suppose that $\calJ \subset M_n$ and $\calK \subset M_m$ are submatricial traceless self-adjoint operator spaces. Suppose $v \subset \C^n$ and $w \subset \C^m$ are bases. There exist graph homomorphisms $\calJ \to \calJ \otimes \calD_w$ and $\calK \to \calD_v \otimes \calK$. In particular, there exist graph homomorphisms $\calJ \to (\calJ \Box \calK)_{v,w}$ and $\calK \to (\calJ \Box \calK)_{v,w}$.
\end{lemma}

\begin{proof}
  Define $\phi: M_n  \to M_n \otimes M_m : X \mapsto \frac{1}{\|w_1\|^2}X \otimes w_1w_1^*$. This map has Kraus operator $E : \C^n \to \C^n \otimes \C^m : v \mapsto v \otimes w_1/\|w_1\|$. Since this Kraus operator is an isometry, we know that $\phi$ is cptp. As well, $\phi(\calJ) = \calJ \otimes w_1w_1^* \subset \calJ \otimes \calD_w$. Similarly, $\calK \to \calD_v \otimes \calK$. Since $\calJ \otimes \calD_w \subset (\calJ \Box \calK)_{v,w}$ and $\calD_v \otimes \calK \subset (\calJ \Box \calK)_{v,w}$, we conclude that $\calJ \to (\calJ \Box \calK)_{v,w}$ and $\calK \to (\calJ \Box \calK)_{v,w}$.
\end{proof}

\begin{theorem}
  Let $\calJ\subset M_n$ and $\calK\subset M_m$ be submatricial traceless self-adjoint operator spaces. Let $v \subset \C^n$ and $w \subset \C^m$ be bases. We have the inequality
  \begin{align*}
    \max\{\chi_{0,cb}(\calJ), \chi_{0,cb}(\calK)\} \leq \chi_{0,cb}((\calJ \Box \calK)_{v,w})\;.
  \end{align*}
\end{theorem}

\begin{proof}
  By lemma~\ref{lemma: Box homomorphism} and by Theorem~\ref{theorem: chi_0 monotone}, we get the inequalities $\chi_{0,cb}(\calJ) \leq \chi_{0,cb}((\calJ \Box \calK)_{v,w})$ and $\chi_{0,cb}(\calK) \leq \chi_{0,cb}((\calJ \Box \calK)_{v,w})$.
\end{proof}

The reverse inequality seems to require the existence of orthogonal bases which colour our submatricial traceless self-adjoint operator spaces. The proof mimicks the proof of Sabidussi's Theorem in \cite{GodEtAl}.

\begin{theorem}
  Let $\calJ \subset M_n$ and $\calK \subset M_m$ be submatricial traceless self-adjoint operator spaces. Let $c = \max\{{\chi}_0(\calJ), \chi_0(\calK) \}$. Suppose that orthonormal bases $v \subset \C^n$ and $w \subset \C^m$ exist for which we have maps $f: [n] \to [c]$ and $g: [m] \to [c]$ for which whenever $f(i) = f(j)$, $v_{f(i)}v_{f(j)}^* \perp \calJ$ and whenever $g(l)= g(k)$, we have $w_{g(l)}w_{g(k)}^* \perp \calK$. We have the inequality
  \begin{align*}
    {\chi}_0((\calJ \Box \calK)_{v,w})\leq \max\{{\chi}_0(\calJ), \chi_0(\calK) \}\;.
  \end{align*}
\end{theorem}

\begin{proof}
  Let $c = \max\{\chi_{0}(\calJ), \chi_{0}(\calK) \}$. Suppose that $v,w $,$f$, and $g$ are as above. Define $h : [n] \times [m] \to [c] : (i,j) \mapsto f(i) + g(j) \mod c$. I claim that whenever $h(i,j) = h(k,l)$, that $(v_i \otimes w_j)(v_k \otimes w_l)^*$ is orthogonal to $(\calJ \Box \calK)_{v,w}$. The identity $h(i,j) = h(k,l)$ tell us $f(i) - f(k) \equiv g(j) - g(l) \mod c$. If $f(i)-f(k) \equiv 0 \mod c$ then we have nothing to check since this means that $f(i) = f(k)$ and $g(j) = g(l)$. Otherwise, $v_iv_k \perp v_sv_s^*$ for all $s$ and $w_jw_l^* \perp w_sw_s^*$ for all $s$. This guarantees that $v_iv_k^* \otimes w_jw_l^*$ is orthogonal to $(\calJ \Box \calK)_{v,w}$.
\end{proof}

\begin{remark}\label{remark: Sabidussi}
  The same proof as above will show us that for some orthonormal bases $v$ and $w$,
  \begin{align*}
    \chi((\calJ \Box \calK)_{v,w}) &\leq \max\{\chi(\calJ), \chi(\calK) \} \text{ and} \\
    \cl{\chi}((\calJ \Box \calK)_{v,w}^\perp) &\leq \max\{\cl{\chi}(\calJ^\perp), \cl{\chi}(\calK^\perp) \}\;.
  \end{align*}
\end{remark}

\begin{corollary}[Sabidussi's Theorem for submatricial traceless self-adjoint operator spaces]
  Suppose that $\calJ \subset M_n$ and $\calK \subset M_m$ are submatricial traceless self-adjoint operator spaces. There exist orthonormal bases $v \subset \C^n$ and $w \subset \C^m$ for which we have the inequalities
  \begin{align*}
    \max\{\chi_{0,cb}(\calJ), \chi_{0,cb}(\calK)\} \leq \chi_{0,cb}((\calJ \Box \calK)_{v,w}) \leq \cl{\chi}((\calJ \Box \calK)^\perp_{v,w}) \leq \max\{\cl{\chi}(\calJ^\perp), \cl{\chi}(\calK^\perp) \}\;.
  \end{align*}
\end{corollary}

\begin{proof}
  By Remark~\ref{remark: Sabidussi} as well as Theorem~\ref{thm: clchi <chi_St} and Corollary~\ref{cor: chi_0 < chi_St}, we get the result.
\end{proof}

\subsection{Hedetniemi's inequality}

The inequality we wish to generalize in this section is a Theorem of Hedetniemi.

\begin{theorem}[Hedetniemi's inequality]
  Suppose that $G$ and $H$ are finite graphs. We have the inequality
  \begin{align*}
    \chi(G \times H) \leq \min\{\chi(G),\chi(H) \}
  \end{align*}
\end{theorem}

This Theorem follows as a special case of the analogous result for $\chi_{0,cb}$, first we generalize the categorical product.

\begin{proposition}
  Let $G$ and $H$ be finite graphs. We have the identity
  \begin{align*}
    \calJ_{G} \otimes \calJ_{H} = \calJ_{G \times H}\;.
  \end{align*}
\end{proposition}

\begin{proof}
  Observe that
  \begin{align*}
    \calJ_G \otimes \calJ_H &= \Span\{E_{i,j} \otimes E_{k,l}: i \sim_G j, k \sim_H l \} \\
    &= \calJ_{G \times H}\;.
  \end{align*}
\end{proof}

We now get a generalization of Hedetniemi's inequality to $\chi_{0,cb}$.
\begin{proposition}
  Suppose that $\calJ \subset M_n$ and $\calK \subset M_m$ are submatricial traceless self-adjoint operator spaces. We have the inequality
  \begin{align*}
    \chi_{0,cb}(\calJ \otimes \calK) \leq \min\{\chi_{0,cb}(\calJ), \chi_{0,cb}(\calK) \}\;.
  \end{align*}
\end{proposition}

\begin{proof}
  The partial trace maps produce graph homomorphisms $\calJ \otimes \calK \to \calK$ and $\calJ \otimes \calK \to \calJ$. By Theorem~\ref{theorem: chi_0 monotone}, we get the inequality.
\end{proof}

\begin{remark}
  A long standing conjecture of Hedetneimi asks whether we get the identity
  \begin{align*}
    \chi(G \times H) = \min\{\chi(G),\chi(H) \}
  \end{align*}
  for any finite graphs $G$ and $H$. As a more general problem, we can ask whether
  \begin{align*}
    \chi_{0,cb}(\calJ \otimes \calK) = \min\{\chi_{0,cb}(\calJ), \chi_{0,cb}(\calK) \}\;.
  \end{align*}
  holds for any submatricial traceless self-adjoint operator spaces $\calJ$ and $\calK$.
\end{remark}

\end{document}